\documentclass{article}
\usepackage{graphicx} % Required for inserting images
\pdfpagewidth
\paperwidth
\pdfpageheight
\paperheight
\usepackage[english]{babel} 
\usepackage{epsfig}
\usepackage{fancyhdr} 
\usepackage{amsmath,amssymb}
\usepackage{amscd} 
\usepackage[T1]{fontenc} 
\usepackage[utf8]{inputenc} 
\usepackage[dvipsnames]{xcolor}
\usepackage{color,listings}
\usepackage{pst-node}
\usepackage{tikz-cd} 

\usepackage{hologo}
\usepackage{geometry}
\usepackage{rotating}
\usepackage{caption}
\usepackage{amsthm}
\usepackage{stackengine,scalerel}
\usepackage{hyperref}
\usepackage{mathrsfs}

\def\F{\mathbb F}
\def\Fq{\mathbb{F}_q}
\def\Fqn{\mathbb{F}_{q^n}}
\def\PG{{\rm{PG}}}

\def\cC{\mathcal{C}}
\newcommand{\qbin}[2]{\genfrac{[}{]}{0pt}{}{#1}{#2}}
\DeclareMathOperator\rk{\mathrm{rk}}
\DeclareMathOperator\Aut{\mathrm{Aut}}
\DeclareMathOperator\GaL{\Gamma\mathrm{L}}
\DeclareMathOperator\PGL{\mathrm{PGL}}
\DeclareMathOperator\PGaL{\mathrm{P}\Gamma\mathrm{L}}
\def\GL{{\rm{GL}}}
\newcommand{\ints}{\mathrm{intn}_\sigma(\Gamma)}

\def\diag{{\mathrm{diag}}}
\def\cB{{\mathcal{B}}}
\def\ptra{\mathscr{A}}

\theoremstyle{plain} 
\newtheorem{thm}{Theorem}[section] 
\newtheorem{cor}[thm]{Corollary} 
\newtheorem{lem}[thm]{Lemma} 
\newtheorem{prop}[thm]{Proposition}

\newtheorem{proposition}[thm]{Proposition} 
\newtheorem{definition}[thm]{Definition}

\title{Scattered polynomials: an overview\\
on their properties, connections and applications}
\author{Giovanni Longobardi}
\date{}

\begin{document}

\maketitle

\begin{abstract}
The aim of this survey is to outline the state of the art in research on a class of linearized polynomials with coefficients over finite fields, known as \textit{scattered polynomials}. These have been studied in several contexts, such as those described in [A.~Blokhuis, M.~Lavrauw. Scattered spaces with respect to a spread in $\mathrm{PG}(n, q)$. \textit{Geom. Dedicata} \textbf{81}(1) 
 (2000), 231--243]  and [G.~Lunardon, O.~Polverino. Blocking sets and derivable partial spreads. \textit{J. Algebraic Combin.} \textbf{14} (2001), 49--56]. Recently, their connection to maximum rank-metric codes was brought to light in [J.~Sheekey. A new family of linear maximum rank distance codes. \textit{Adv. Math. Commun.} \textbf{10}(3) (2016), 475--488.]. This link has significantly advanced their study and investigation, sparking considerable interest in recent years.\\
Here, we will explore their relationship with certain subsets of the finite projective line $\PG(1,q^n)$ known as \textit{maximum scattered linear sets}, as well as with codes made up of square matrices of order $n$ equipped with the rank metric. We will review the known examples of scattered polynomials up to date and discuss some of their key properties. We will also address the classification of maximum scattered linear sets of the finite projective line 
$ \PG(1,q^n)$ for small values of $n$ and discuss characterization results for the examples known so far.\\
Finally, we will retrace how each scattered polynomial gives rise to a translation plane, as discussed in [V. Casarino, G. Longobardi, C. Zanella. Scattered linear sets in a finite projective line and translation planes, \textit{Linear Algebra Appl.} \textbf{650} (2022), 286–298] and in [G. Longobardi, C. Zanella. A standard form for scattered linearized polynomials and properties of the related translation planes, \textit{J. Algebr. Comb.} \textbf{59}(4) (2024), 917–937].

\end{abstract}

\medskip 


\section{Introduction and Preliminaries}
Let $q=p^e$  where $p$ is a prime number and let $\F_{q^n}$ be the finite field of order $q^n$. A \textit{linearized polynomial}, or a $q$-$polynomial$, with coefficients over $\F_{q^n}$ is a polynomial of the form
\begin{equation*}
f:=f(X)=\sum_{i=0}^{\ell} c_i X^{q^i},
\end{equation*}
where $c_i \in \F_{q^n}$ and $\ell$ is a non-negative integer. We will denote the set of these polynomials by $\mathscr{L}_{n,q}[X]$. If $\ell$ is the largest integer such that $c_\ell \neq 0$, we say that $\ell$ is the $q$-$degree$ $\deg_q(f)$ of $f$. It is straightforward to show that  if $f$ is a linearized polynomial the map $ a \in \F_{q^n} \longmapsto f(a) \in \F_{q^n}$ is  an $\F_q$-linear endomorphism of $\F_{q^n}$, when $\F_{q^n}$ is seen as a vector space over $\F_q$.\\
In \cite{Sheekey2016}, Sheekey introduced a class of linearized polynomials called scattered. More precisely, a \textit{scattered linearized polynomial} over $\F_{q^n}$ is an element $f \in \mathscr{L}_{n,q}[X]$ such that 
for any $y, z\in\Fqn^*$, the condition
\begin{equation}\label{scatt-property}
\frac{f(y)}{y}=\frac{f(z)}{z}
\end{equation}
implies that $y$ and $z$ are $\Fq$-linearly dependent.\\
Although scattered linearized polynomials appeared in different contexts and have been studied for about twenty years because of their interest in various combinatorial contexts, such as blocking sets, planar spreads and translation planes \cite{blokhuis_scattered_2000,B,Lunardon_Polverino, Ostrom}, their link with the theory of rank-metric codes gave a new gain to their investigation. In this survey we will outline the state of the art of this research and recall some of their applications.

\subsection{Rank-metric codes} \label{RM-codes}
Coding theory is the branch of mathematics and computer science that deals with the design of codes for the efficient and reliable transmission of information across noisy communication channels. The primary goal is to create methods that allow messages to be transmitted in such a way that errors introduced during transmission (due to noise, interference, or other issues) can be detected and corrected, ensuring that the original message is accurately received.
Error-correcting codes are specially designed codes that allow the receiver to correct errors without needing to retransmit the message. Here, the key concept is the use of a distance function, which measures how different a received message is from valid codewords, enabling error correction by mapping the received (possibly corrupted) message to the closest valid codeword.

The codes in \textit{Hamming metric} are certainly the most widely studied and used. Here, the codewords are vectors with entries over a finite field and the distance between two of them is determined by counting how many positions they differ, for more details see \cite{McWilliams_Sloane}. 
Another approach is rank-metric coding proposed in \cite{Kotter_Kschischang}, where codewords are matrices. In this case, the distance between two codewords is defined as the rank of their difference.
This is very useful in encoding through different types of channels, including networks. \\

In the following we describe these codes and their representation in the linearized polynomials setting. Throughout this paper, if $S$ is a set of field elements (or vectors), we denote by $S^*$ the set of non-zero
elements (non-zero vectors) of $S$.
Let $\Fq$ be the finite field of $q$ elements and let $\mathbb{F}_q^{m \times n}$ be the set of $m \times n$ matrices with entries over $\Fq$ endowed with \textit{rank distance} 
\begin{equation}\label{rankdistance}
    d_{\rk}: (A,B) \in \mathbb{F}_q^{m \times n} \times \mathbb{F}_q^{m \times n} \longrightarrow \rk(A-B) \in \{0,1,\ldots, \min\{m,n\}\}.
\end{equation}
A subset of $\mathbb{F}_q^{m \times n}$, containing at least two elements, is called a {\em rank-metric code}, \textit{RM}-\textit{code} shortly. The {\em minimum distance} of a code $\mathcal{C}$ is defined by 
\begin{equation*}
   d_{\rk}(\cC)= \min_{\substack{A, B \in \cC\\A \neq B}} d_{\rk}(A,B).
\end{equation*}
If $d=d_{\rk}(\mathcal{C})$, $\mathcal{C}$ is said to be a rank-metric code with parameters $(m, n, q;d)$.
If $\cC$ is an $\Fq$-linear subspace of $\Fq^{m\times n}$, then $\cC$ is called $\Fq$-\textit{linear} and its
dimension $\dim_{\Fq}\cC$ is defined to be the dimension of $\cC$ as  subspace over $\Fq$.

For the applications in classical coding theory, given the positive integers $m, n$ and $1 \leq d \leq \min\{m, n\}$, it is desirable to have rank-metric codes which are
as large as possible in size \cite{Gabidulin_ Paramonov_Tretjakov, Kotter_Kschischang, Kotter_Kschischang_Silva}.
In \cite[Theorem 5.4]{Delsarte}, Delsarte proved that the size of a rank-metric code with parameters $(m,n,q;d)$ must satisfy an upper
bound, the so-called \textit{Singleton-like bound}. Precisely,
\begin{equation} \label{singleton}
\left \vert \cC\right \vert \leq q^{\max\{m,n\}(\min\{m,n\}-d+1)}.
\end{equation}
If the size of the code $\cC$ meets this bound, then $\cC$ is called \textit{maximum rank distance
code}, or \textit{MRD} \textit{code} for short. 
The \textit{adjoint code} of $\mathcal{C}$ is defined as the code made up of the transposes of matrices belonging to $\mathcal{C}$:
\begin{equation*}
    \mathcal{C}^t=\{C^t: C \in \mathcal{C}\},
\end{equation*}
where the superscript $t$ stands for the matrix transposition.
Two $\Fq$-linear rank-metric codes $\mathcal{C}, \mathcal{C}^{\prime} \subseteq \mathbb{F}_q^{m \times n}$, $m,n \geq 2$, are called {\em equivalent}, in symbols $\cC \cong \cC'$, if there exist $P\in \GL(m, q)$, $Q \in \GL(n,q)$ and a field automorphism $\rho \in \Aut(\mathbb{F}_q)$ such that
$$\mathcal{C}^{\prime}=P\mathcal{C}^{\rho}Q=\left\{ PC^{\rho}Q: C \in \mathcal{C} \right\}.$$
 
When $m=n$, in addition to being equivalent, two codes are said \textit{adjointly equivalent} if
$$\mathcal{C}^{\prime}=P(\mathcal{C}^t)^{\rho}Q=\left\{ PC^{\rho}Q: C \in \mathcal{C}^t \right\}.$$

In general, determining whether two RM-codes are equivalent is challenging. The idealisers provide a useful criterion for addressing this equivalence problem.
More precisely, the \textit{left idealiser} and \textit{right idealiser} of a rank-metric code $\cC \subseteq \Fq^{m \times n}$ are defined as
\[I_L(\cC) =\{X \in \F_{q}^{m \times m} : X C \in \cC \text{ for all } C\in \cC \}, \]
and
\[I_R(\cC) =\{Y \in \F_{q}^{n \times n}: C Y \in \cC \text{ for all }C \in \cC \}, \]
respectively. Indeed, in \cite[Proposition 3.1]{Lunardon_Trombetti_Zhou},  the authors proved that equivalent codes have equivalent idealisers. Moreover, in \cite[Theorem 5.4]{Lunardon_Trombetti_Zhou}, they showed that if $m \leq n$ and $\cC$ is an $\Fq$-\textit{linear} MRD code with minimum distance $d$, then $I_L(\cC)$ is a field with $ q \leq |I_L(\cC) |\leq q^m$, and if  $\max\{d, m - d + 2\} \geq  \left \lfloor  \frac{n}{2} \right \rfloor+ 1$, then its  right idealiser $I_R(\cC)$ is a field as well, with $q \leq |I_R(\cC)| \leq q^n$, see also   \cite[Result 3.4]{Polverino_Zullo}. Actually, the following holds.

\begin{thm} \label{idealizers} \cite[Theorem 3.1]{Longobardi_Zanella2024} Let $\cC \subseteq \F_q^{m \times n}$ be an $\F_q$-linear rank-metric code. If $I_L(\cC)$  (resp.\ $I_R(\cC)$) is a field, then it is isomorphic to a subfield of $\F_{q^m}$ (resp.\ $\F_{q^n})$. \end{thm}

\begin{proof}
Let us suppose that $\mathcal{M}:=I_L(\cC)$ is a matrix field. Since  any scalar matrix $\lambda I_m$, where $\lambda \in \F_q$ and $I_m$ the identity matrix of order $m$,  is contained in $\mathcal{M}$, we have that $|\mathcal{M}|=q^a$ for some $1 \leq a \leq  m$. Moreover, $\F_q^m$ is a (right) vector space over $\mathcal M$ with exterior product $ \mathbf x M$, $\mathbf x\in\F_q^m$ and $M\in\mathcal M$. Let  $b=\dim_{\mathcal{M}}\F^m_{q}$, then we get
$$q^m=\vert \F^m_{q}\vert=\vert \mathcal{M} \vert^{\dim_{\mathcal{M}}\F^m_{q}}=q^{ab}.$$
Hence, $a$ divides $m$ and this concludes the proof.
\end{proof}

Finally, the \textit{weight} of a codeword $C \in  \cC$ is defined as the rank of $C$. In \cite{Delsarte},
Delsarte  (and later Gabidulin in \cite{Gabidulin}) determined the weight
distribution of an MRD code. 
Before stating this result, we recall the following definition. Let $n,k$ be  non-negative integers with $k \leq n $. The \textit{Gaussian binomial
coefficient} of $n$ and $k$ is defined as follows:
\begin{equation*}
\qbin{n}{k}_q=
\begin{cases} 
\frac{(q^n-1)\cdots(q^{n-k+1}-1)}{(q^k-1)\cdots(q-1)}  \hspace{1.3cm}\textnormal{ if $k>0$} \\
\hspace{1.5cm} 1 \hspace{2.5cm}\textnormal{otherwise.} 
\end{cases}
\end{equation*}

Let us denote by $A_i(\cC)$ the number of codewords of weight $i$ of an RM-code $\cC$.

\begin{thm}\label{weightdistribution}\cite[Theorem 5.6]{Delsarte}
Let $\cC \subseteq \F_{q}^{m \times n}$ be an MRD code with minimum distance $d$ and
suppose $m \leq n$. Then,
\begin{equation}
A_{d+\ell}(\cC) = 
\qbin{m}{d + \ell}_q
\sum_{j=0}^\ell (-1)^{j-\ell} \qbin{\ell + d}{\ell- j}_q q^{\binom{\ell-j}{2}}(q^{n(j+1)} -1), \qquad \ell \in \{0, 1,\ldots, m -d\}.
\end{equation}
In particular, $A_{d+\ell}(\cC) >0$ for any $\ell \in \{0, 1,\ldots, m -d\}$ and the number of codewords with minimum
weight  is
\begin{equation}
A_d(\cC) = \qbin{m}{d}_q (q^n-1)= (q^n-1) \frac{(q^m-1)(q^{m-1}-1) \cdots (q^{m-d+1}-1)}{(q^d-1)(q^{d-1}-1) \cdots (q-1)}.
\end{equation}
\end{thm}
\subsection{Rank-metric codes as subsets of linearized polynomials} 
As said before, any linearized polynomial $f \in \mathscr{L}_{n,q}[X]$ induces an $\F_q$-linear endomorphism of $\F_{q^n}$ and it is well known that there is a one-to-one correspondence  between an element of $\mathrm{End}_{\F_q}(\F_{q^n})$	and a linearized polynomial with $q$-degree at most $n-1$, see \cite{Lidl_Niederreiter}.\\
Throughout this article, we will denote by $\mathscr{\tilde{L}}_{n,q}[X]$ the set of linearized polynomials with coefficients over $\F_{q^n}$ and $q$-degree at most $n-1$. The algebraic structure $(\mathscr{\tilde{L}}_{n,q}[X],+,\circ,\cdot)$, where $+$ is the addition of polynomials, $\circ$ is the composition of polynomials  modulo $X^{q^n}-X$  and $\cdot$ is the scalar multiplication by an element of $\F_q$, is isomorphic to the $\F_q$-algebra $\mathrm{End}_{\F_q}(\F_{q^n})$. 
%and, by fixing an $\F_q$-basis of $\F_{q^n}$, to the algebra of the matrices $\F_q^{n \times n}$, as well.\\
Let $f \in \mathscr{\tilde{L}}_{n,q}[X]$, the set of the roots  and  the values set of $f$  are both $\F_q$-subspaces of $\F_{q^n}$. These are called  the \textit{kernel} and the \textit{image} of $f$ and they are denoted by $\ker f$ and $\mathrm{im} f$, respectively.
We will define  \textit{rank} of $f$, $\rk f$ in symbols, as
the dimension over $\F_{q}$ of $\mathrm{im} f$. Clearly, if $\deg_q(f) = k$, the rank of $f$ is at least $n-k$. This follows from the fact that such a polynomial may have at most $q^k$ roots, and hence its kernel has dimension at most $k$, implying that $\mathrm{im} f$ has at least $q^{n - k}$ elements.
This turns out to be a special case of the following more general result.
\begin{thm}\cite[Theorem 5]{gow-quinlan_2} \label{qspolynomial} Let $\mathbb{L}$ be a cyclic Galois extension of a field $\F$ of degree $n$, and suppose that $\sigma$ generates the group  $\mathrm{Gal}(\mathbb{L}\vert  \F)$. Let $k$ be an integer satisfying $1 \leq k < n$, and let $c_0, c_1,\ldots, c_{k-1}$ be elements of
	$\mathbb{L}$, not all equal to zero. Then the polynomial 
	\begin{equation*}
	f = c_0X + c_1X^\sigma + \ldots + c_{k-1}X^{\sigma^{k}}
	\end{equation*}
	has rank at least $n - k$.
\end{thm}
Taking $\mathbb{L} = \F_{q^n}$, $\F = \F_q$, and $ \sigma : x \in \F_{q^n} \longmapsto x^q \in \F_{q^n}$ returns the above statement about linearized polynomials, while if $\sigma : x \in \F_{q^n} \longmapsto x^{q^s} \in \F_{q^n}$
for some $s \in \{1,\ldots,n-1\} $ such that $\gcd(s,n)=1$, then the polynomial $f$ becomes what is known as a $\sigma$-\textit{linearized polynomial}, or a $q^s$-\textit{polynomial}.
As before, the set $\tilde{\mathscr{L}}_{n,q,s}[X]$ of $\sigma$-linearized polynomials can also be structured as an $\F_q$-algebra isomorphic to $\mathrm{End}_{\F_q}(\F_{q^n})$.\\
Given a $q$-polynomial $f=\sum_{i=0}^{n-1} c_iX^{q^i} \in \tilde{\mathscr{L}}_{n,q}[X]$, the associated \textit{Dickson matrix} (or $q$-\textit{circulant matrix}) $D_f$ is defined as
\begin{equation}
D_f=
\begin{pmatrix}
    c_0 & c_1 & \ldots & c_{n-1} \\
    c_{n-1}^q & c_0^q & \ldots  & c_{n-1}\\ 
    \vdots & \vdots & \vdots & \vdots \\
    c_{1}^{q^{n-1}} & c_{2}^{q^{n-1}} & \ldots & c_{0}^{q^{n-1}}
\end{pmatrix}.
\end{equation}
The rank of the matrix $D_f$ is equal to the rank of $f$, see \cite[Proposition 4.4]{revisited}. It is straightforward to see that $\tilde{\mathscr{L}}_{n,q}[X]$ is isomorphic to the $\F_q$-algebra of Dickson matrices $\mathscr{D}_{n,q}$ with entries over $\F_{q^n}$. Let us fix  $\mathscr{B}=\{\beta_i\}^{n-1}_{i=0}$ an $\F_q$-basis of $\F_{q^n}$, the map
$\Psi_\mathscr{B}: f  \longmapsto B^{-1} D_f B$,
 where 
 \begin{equation}\label{B}
B=\begin{pmatrix}
\beta_0 & \beta_1 & \ldots & \beta_{n-1}\\
\beta^q_0 & \beta^q_1 & \ldots & \beta^{q}_{n-1}  \\
\vdots & \vdots  & \vdots & \vdots \\
\beta_0^{q^{n-1}} & \beta_1^{q^{n-1}} & \ldots & \beta_{n-1}^{q^{n-1}}
\end{pmatrix},
\end{equation}
 is an isomorphism from $\tilde{\mathscr{L}}_{n,q}[X]$ to $\F_{q}^{n \times n}$, see \cite[Lemma 4.1]{revisited}.

In light of the above, the rank distance introduced in \eqref{rankdistance} can be read in $\mathscr{\tilde{L}}_{n,q}[X]$  as
$$d_{\rk}(f_1,f_2)=\rk(f_1-f_2)$$
with $f_1,f_2 \in \mathscr{\tilde{L}}_{n,q}[X]$. After fixing an $\F_q$-basis of $\F_{q^n}$, any rank-metric code $\cC \subseteq \F_q^{n \times n}$ can be seen  as a subset of
linearized polynomials of $\tilde{\mathscr{L}}_{n,q}[X]$.
Although also codes whose codewords are rectangular matrices can be seen in a linearized polynomial setting (\cite[Section 1.2]{Longobardi_Thesis} and \cite{Polverino_Zullo}) for the purposes of this survey,  we will deal only with square rank-metric codes.
We can reformulate the notions recalled in Subsection \ref{RM-codes} in terms of $ q $-polynomials with coefficients over $\F_ {q ^ n} $.
An $\F_q$-linear rank-metric code $\mathcal{C}$ is an $\F_q$-subspace of the vector space $\tilde{\mathscr{L}}_{n,q}[X]$ equipped with the above mentioned rank metric and its minimum distance is defined as
\begin{equation}
d_{\rk}(\mathcal{C}) = \min_{\substack{f \in \mathcal{C}^*}} \rk f. 
\end{equation}
Let us denote by $\mathscr{B}^*=\{\beta_i^*\}_{i=0}^{n-1}$ the dual basis of $\mathscr{B}$, see \cite[Definition 2.30]{Lidl_Niederreiter}. The transpose operation $\cdot^t$ in $\F_{q}^{n \times n}$ can be read  as the map $\widehat{\cdot}:\tilde{\mathscr{L}}_{n,q}[X] \longrightarrow \tilde{\mathscr{L}}_{n,q}[X]$ which makes commutative the following diagram
\[
 \begin{CD} \tilde{\mathscr{L}}_{n,q}[X] @>\widehat{\cdot}>> \tilde{\mathscr{L}}_{n,q}[X]  \\ @VV\large{\Psi_\mathscr{B}}V @VV\Psi_{\mathscr{B}^*}V\\ {\F_q}^{n \times n}  @> \cdot^t>> \F_{q}^{n \times n} \end{CD}
 \]
Then, if $f=\sum_{i=0}^{n-1} c_i X^{q^i} \in   \tilde{\mathscr{L}}_{n,q}[X]$, the \textit{adjoint polynomial} of $f$ is defined as
\begin{equation*}
    \hat{f}=\sum_{i=0}^{n-1}c^{q^{n-i}}_iX^{q^{n-i}}
\end{equation*} and it corresponds in $\tilde{\mathscr{L}}_{n,q}[X]$ to the transpose of the matrix $B^{-1}D_fB$  (cf. \eqref{B}). Let $\mathcal{C} \subseteq \tilde{\mathscr{L}}_{n,q}[X] $ be a rank-metric code, then the adjoint code of $\cC$ corresponds to the set $ \widehat{\cC}=\{ \hat{f} : f \in \mathcal{C}\}$. Moreover, it is routine to verify that $\hat{\hat{f}}=f$.\\
Finally, two $\F_q$-linear rank-metric codes $\mathcal{C}$ and $\mathcal{C}^\prime$ are equivalent or adjointly equivalent if  there
exists $(\alpha,\rho,\beta)$ such that $\alpha,\beta \in \tilde{\mathscr{L}}_{n,q}[X]$ with $\rk \alpha = \rk \beta =n$, and $\rho \in \Aut(\mathbb{F}_q)$ such that
\begin{center}
$\mathcal{C}^\prime = \{  \alpha \circ f^\rho \circ \beta \colon  f \in \mathcal{C}\} $ \,\,or\,\, $\mathcal{C}^\prime = \{\alpha \circ f^\rho \circ \beta \colon  f\in \widehat{\mathcal{C}}\},$
\end{center}
respectively, where the automorphism $\rho$ acts only over the coefficients of a polynomial $f$ in $\cC$ or $\widehat{\cC}$, respectively.
If $\cC=\cC'$, the set of the triples defined as above  is a group which is called
the \textit{automorphism group} of $\cC$ and it is denoted by $\Aut(\cC)$.
Also the notion of left and right idealiser can be rephrased in this setting as follows  
\[I_L(\cC) =\left \{\varphi\in \tilde{\mathscr{L}}_{n,q}[X]: \varphi\circ f\in \cC \text{ for all }f\in \cC \right \} \]
and
\[I_R(\cC) = \left \{\varphi\in \tilde{\mathscr{L}}_{n,q}[X]: f\circ \varphi \in \cC \text{ for all }f\in \cC \right \}, \]
respectively. By \cite[Corollary 5.6]{Lunardon_Trombetti_Zhou} and Theorem \ref{idealizers}, if $\cC$ is an MRD code both the idealisers are isomorphic to subfields of $\F_{q^n}$.

\subsection{Linear sets of a finite projective space} \label{linearsets}

Let $V=V(r,q^n)$ be a vector space over $\F_{q^n}$ of dimension $r$ and let $\Lambda=\PG(V,q^n)=\PG(r-1,q^n)$ be the projective space associated with it. A set of points $L_U$ of $\Lambda$  is called an $\F_q$-\emph{linear set} of \emph{rank} $\nu$ if it consists of the points defined by the non-zero vectors of a $\nu$-dimensional $\F_q$-subspace $U$ of $V$, in symbols
\[L_U=\left\{\langle \mathbf{u} \rangle_{\Fqn} :   \mathbf{u}\in U^*\right\}. \]
The term `linear' was first introduced by Lunardon \cite{lunardon_normal_1999} who considered a particular class of blocking sets. In the past two decades, in addition to blocking sets, linear sets have been intensively investigated and applied to construct and characterize various objects in finite geometry such as two-intersection sets, translation spreads of the Cayley generalized Hexagon, translation ovoids of polar spaces, semifields and rank-metric codes. The reader may refer to the surveys 
\cite{Lavrauw_Polverino,LavrauwVanderVoorde, polverino_linear_2010, Polverino_Zullo,Sheekey2019}
and the references therein.\\
An $\F_q$-linear set $L_U$ could be represented by different $\F_q$-subspaces of $V$. Indeed, it is clear that $L_{\lambda U}=L_{U}$ for any $\lambda \in \F_{q^n}^*$, where $\lambda U=\{\lambda \mathbf{u}: \mathbf{u} \in U\}$.\\ Examples of linear sets are the subgeometries of a finite projective space, see \cite{Hirschfeld}. More precisely, a \textit{(canonical) subgeometry}  $\Sigma$ of $\Lambda$ is a linear set  with rank $r$ and  such that $\langle \Sigma \rangle =\Lambda$. %It follows that a linear set is a (canonical) subgeometry  of $\Lambda$ if and only if
%any of its frame is also a frame of $\Lambda$.\\
Two linear sets $L_U$ and
$L_W$ of $\Lambda = \PG(r - 1, q^n)$ are $\PGaL$-\textit{equivalent} 
(or simply \textit{equivalent}) if there exists  a collineation $\kappa \in  \PGaL(r, q^n)$ such that $L^{\kappa}_U=\kappa(L_U)=L_W$. Clearly, if $U$ and $W$ are $\Fq$-subspaces of $V$ in
the same $\Gamma \mathrm{L}(r, q^n)$-orbit, then $L_U$ and $L_W$ are equivalent. Indeed, if $\gamma  \in \GaL(r,q^n)$ such that $U^\gamma= W$ then, the collineation $\kappa_\gamma$
induced by $\gamma$ is a $\PGaL$-equivalence that maps $L_U$ into $L_W$.
However, this is not a necessary condition. Indeed, a linear set $L_U$ could be represented by different $\F_q$-subspaces which are not in the same $\GaL(r, q^n)$-orbit. If this is not the case, a linear set $L_U$ of rank $\nu$ is a said to be \textit{simple}. Hence, $L_U$ is simple if for each $\F_q$-subspace $W$ of $V$
such that $\dim_{\F_q} W=\nu$ and $L_U = L_W$, the subspaces $U$ and $W$ are in the same $\GaL(r, q^n)$-orbit.\\
 A collineation $ \kappa \in \PGaL(r,q^n)$ such that $L^\kappa=L_U=L$ is called an \emph{automorphism} of $L$ and the set of all such maps is called the \emph{automorphism group} of $L$ and it is denoted by $\Aut(L)$. This is a subgroup of $\PGaL(r,q^n)$.\\
The most studied linear sets are those that satisfy being maximal in size. Clearly, any $\Fq$-linear set in $\Lambda=\PG(r-1, q^n)$ of rank greater than $n(r-1)$ coincides with $\Lambda$. Moreover, it is straightforward to see that
\begin{equation*}
    |L_U|\leq \frac{q^{\nu}-1}{q-1}=q^{\nu-1}+q^{\nu-2}+\ldots+q+1.
\end{equation*} When the equality is achieved, $L_U$ and the underlying vector space $U$ are called \textit{scattered}. A scattered linear set $L_U$ of $\Lambda$ with largest possible rank  is called  \emph{maximum scattered linear set}.  In \cite{blokhuis_scattered_2000}, Blokhuis and Lavrauw  proved  a bound on  the largest possible rank for a scattered linear set.

\begin{thm} \label{boundscattered} \cite[Theorems 2.1 and 4.3]{blokhuis_scattered_2000} Let  $L_U$ be a maximum scattered
$\Fq$-linear set of rank $\nu$ in $\Lambda = \PG(r -1, q^n)$, then if $r$ is even
\[\nu =\frac{rn}{2},\]
otherwise
\[\frac{n(r-1)}{2} \leq  \nu \leq \frac{rn}{2}.\]
\end{thm}

Actually, in \cite{blokhuis_scattered_2000} the results above are stated in terms of subspaces and spreads of a given vector space. We will briefly recall these notions. Let $ V (m, q)$ denote an $m$-dimensional $\F_q$-vector space. A \textit{partial} $t$-\textit{spread} of $V(m,q)$ is a set $\mathcal{S}$ of $t$-dimensional $\F_q$-subspaces meeting pairwise in the null space of $V$. If every non-zero vector of $V(m,q)$
is contained in exactly one element of $\mathcal{S}$, this is simply called a $t$-\textit{spread} of $V(m,q)$. In \cite{Segre}, Segre showed that a $t$-spread of $V(m,q)$ exists if and only if $t$ divides $m$.
If we consider $V=V(r,q^n)$, this can be seen  as an $rn$-dimensional  vector space over $\Fq$ and it is well-known that the one-dimensional $\Fqn$-subspaces of $V$, viewed as $n$-dimensional $\Fq$-subspaces form an $n$-spread of $V$.
This spread $\mathcal{D}$ is called the \textit{Desarguesian spread}.
Hence, an $\F_q$-linear set $L_U \subseteq \PG(r-1,q^n)$ is maximum scattered if and only if the underlying $\F_q$-vector space $U$ is maximum scattered with respect to $\mathcal{D}$. This means that $U$ meets every
element of $\mathcal{D}$ in an $\Fq$-subspace of dimension at most one:
\begin{equation*}
    \dim_{\F_q}(U \cap \langle \textbf{v} \rangle _{\Fqn}) \leq 1, \quad \quad \forall \, \textbf{v} \in V
\end{equation*}  
and the dimension of $U$ is maximal with respect to the property of being scattered. In this paper, for a subspace to be scattered will always mean scattered with respect to the Desarguesian spread $\mathcal{D}$ of $V$. A generalization of the concept of scattered subspace can be found in \cite{evasive} and \cite{h-scattered}.\\

We conclude this section by giving the following result that characterizes the coordinates of the points belonging to linear sets of $\PG(r-1,q^n)$ of rank $n$.

\begin{proposition} \cite[Lemma 7]{mrdcodeslinearsets} $\&$ \cite[Lemma 2.2]{SheekeyVandeVoorde} \label{prop:rankn} Let $L=L_U$ be an $\Fq$-linear set of rank $n$ in $\Lambda=\PG(r - 1, q^n)$. Then there exist $f_0,f_1,\ldots,f_{r-1} \in \tilde{\mathscr{L}}_{n,q}[X]$ such that

\begin{equation}\label{rankn}
L = \left \{\langle (f_0(x),f_1(x),\ldots, f_{r-1}(x)) \rangle_{\Fqn} \colon x  \in  \Fqn^* \right \},
\end{equation}
with $\bigcap_{i=0}^{r-1} \mathrm{ker} f_i= \{\boldsymbol{0}\}$.
 If $\langle L \rangle  = \Lambda$, then the polynomials $f_i$  are
linearly independent over $\Fqn$.\\
Vice versa, if $f_i \in \tilde{\mathscr{L}}_{n,q}[X]$, $i \in \{0,\ldots,r-1\}$ with $\bigcap_{i=0}^{r-1} \mathrm{ker} f_i= \{\boldsymbol{0}\}$, then
the point set in \eqref{rankn} is an $\Fq$-linear set of rank $n$ in $\Lambda$.
\end{proposition}
\begin{proof}
Since $L_U$ is an $\Fq$-linear set of rank $n$ in $\Lambda$, the $\Fq$-subspace $U$ is isomorphic to $\Fqn$ as an $\Fq$-vector space. Then, let $\Phi : U \longrightarrow \Fqn$ be an isomorphism and  consider

 $$\left \{\langle \Phi(x) \rangle_{\Fqn} \colon x \in \Fqn^*\right\}.$$ Now, $\Phi(x)$ belongs to $\F_{q^n}^r$ and hence, it can be written as $(f_0(x),\ldots, f_{r-1}(x))$ where $f_i \in \tilde{\mathscr{L}}_{n,q}[X]$, $i=0,\ldots,r-1$.
If there is a non-zero element in the intersection of the kernels of $f_i$'s, then the rank of
$L$ is strictly less than $n$, a contradiction.
Finally,  if the maps $f_0,\ldots, f_{r-1}$ are linearly dependent over $\Fqn$,
then there exist $\alpha_i \in \Fqn$ , not all zero, such that 
$$\sum_{i=0}^{r-1}\alpha_i f_i(x) = 0$$
for all $x\in \F_{q^n}$, implying that
$L$ is contained in the hyperplane with equation $\alpha_0X_0+\alpha_1X_1+\ldots+\alpha_{r-1}X_{r-1}=0$, and hence, $\dim \langle L \rangle < r-1$.
Vice versa, if all $f_i$ are $q$-polynomials, then $\{ (f_0(x),\ldots, f_{r-1}(x)) \colon x \in \Fqn \}$ defines an
$\Fq$-subspace $U$ of rank $n$ if and only if $\ker f_0 \cap\ldots \cap \ker f_{r-1} = \{\boldsymbol{0}\}$, implying that $L=L_U$ is an $\Fq$-linear set of rank $n$.
\end{proof}

\section{Scattered \texorpdfstring{$\mathbb{F}_q$}{TEXT}-subspaces of \texorpdfstring{$\mathbb{F}_{q^n}^2$}{TEXT} and linearized polynomials}

From now on, we will focus on maximum scattered $\Fq$-subspaces of $\F_{q^n}^2$. These subspaces are also significant because, via the direct sum construction, they allow us to obtain new maximum scattered subspaces in higher dimensions, see \cite[Section 3]{BaGiuMaPo}. A linear set associated with a subspace of this type turns out to be a point set of  the finite projective line $\PG(1,q^n)$ and  by Theorem \ref{boundscattered}, it has rank equal to $n$. Since $\PGL(2,q^n)$ acts 3-transitively over the points of $\PG(1,q^n)$, any linear set of $\PG(1,q^n)$ of rank $n$ is equivalent to an $L_U$ such that the point $\langle (0, 1) \rangle_{\Fqn} \not \in L_U$. Therefore, by Proposition \ref{prop:rankn}, there exists a linearized
polynomial $f\in \tilde{\mathscr{L}}_{n,q}[X]$ such that 
\[ U=U_f :=\{(x,f(x)): x\in \F_{q^n}  \}, \]
and $L_U=L_{U_f}$. We will use $L_f$ to denote the linear set defined by $U_f$.

\begin{thm} \cite[Section 5]{Sheekey2016}
Let $f \in \tilde{\mathscr{L}}_{n,q}[X]$. Then, the following are equivalent:
\begin{enumerate}
\item [$i)$] $U_f$ is maximum scattered subspace of $\F_{q^n}^2$;
\item [$ii)$] $f$ is scattered;
\item [$iii)$] The set of $q$-polynomials
\begin{equation}\label{Cf}
\mathcal{C}_f=\langle X, f(X) \rangle_{\F_{q^n}}= \left \{ aX+bf(X) \colon a,b \in \F_{q^n} \right\} \subset \tilde{\mathscr{L}}_{n,q}[X]
\end{equation}
is a maximum rank distance code with parameters $(n,n,q;n-1)$.
\end{enumerate}
\end{thm}
\begin{proof}
$ i) \implies ii)$.  Suppose that $f(x)/x=f(y)/y$ for some $x,y\in\F_{q^n}^*$. The vector  $(y,f(y))\in U_f \cap\langle(x,f(x))\rangle_{\F_{q^n}}$. Since $U_f$ is scattered, $\dim_{\F_q}(U_f\cap\langle(x,f(x))\rangle_{\F_{q^n}})\leq1$. Therefore $(y,f(y)) =\lambda(x,f(x))$ for some $\lambda \in \F_q$, giving that $x$ and $y$ are $\F_q$-linearly dependent. Hence, $f$ is a scattered polynomial.\\
$ii)\implies i)$. Let us suppose by contradiction that $U_f$ is not scattered, i.e., there exists some $x\in\F_{q^n}^*$ such that $\dim_{\F_q}(U_f\cap\langle(x, f(x))\rangle_{\F_{q^n}})>1$. Then, if $y\in\F_{q^n}^*$ satisfies $(y,f(y))\in U_f\cap\langle(x, f(x))\rangle_{\F_{q^n}}$, there exists $\lambda\in\F_{q^n}\setminus\F_q$ such that  $y= \lambda x$ with $f(x)/x=f(y)/y$. This is a contradiction to the assumption that $f$ is a scattered polynomial.\\
$ii) \implies iii)$. It is easy to observe that $|\cC_f|=q^{2n}$. In order to be a maximum rank distance code, by the Singleton-like bound (cf. \eqref{singleton}), $d_{\rk}(\cC_f)$ must be $n-1$. Then, it is enough to prove that $\dim_{\F_q}\ker(aX+bf(X)) \in \{0,1\}$ for every $(a,b) \in \F^2_{q^n} \setminus \{(0,0)\}$. Clearly, if $b=0$, then $\dim_{\F_q} \ker aX=0$ for any $a \in \F_{q^n}^*$. If $b\neq0$,  proving  that $\dim_{\F_q}\ker(aX+bf(X)) \in \{0,1\}$ for every $(a,b) \in \F_{q^n} \times \F_{q^n}^*$ is equivalent to showing that $\dim_{\F_q}\ker(mX +f(X)) \leq 1$ for every $m \in \F_{q^n}$.\\
Let $\overline{x}, \overline{y}\in \ker (mX+f(X))$. Then, $m\overline{x}+f(\overline{x})=m\overline{y}+f(\overline{y})=0$. This implies that $f(\overline{x})/\overline{x}=f(\overline{y})/\overline{y}$ and hence $\overline{y}=\lambda \overline{x}$ for some $\lambda\in\F_q^*$. This gives $\dim_{\F_q} \ker(mX+f(X))=1$ and  the result follows.\\
$iii)\implies ii)$. This follows directly by noting that $d_{\rk}(\cC_f)=n-1$ implies $\dim_{\F_q}\ker(mX+f(X))$ is either 0 or 1. Then, if $f(x)/x=f(y)/y=-m$  for some $m \in \F_{q^n}$ and $x,y \in \F_{q^n}^*$, we get $y \in \langle x \rangle_{\F_q}$. This concludes the proof.
\end{proof}

By using Proposition \ref{weightdistribution}, we have the weight distribution of an MRD code $\cC_f$ defined as in \eqref{Cf}.

\begin{prop}
    Let $f \in \tilde{\mathscr{L}}_{n,q}[X]$ be a scattered polynomial and consider the MRD code $\cC_f$. Then,
\begin{equation}
    A_0(\cC_f)=1, \quad \quad A_{n-1}(\cC_f)=q(q^n-2q^{n-1}+1)\frac{q^{n}-1}{q-1}, \quad \quad A_n(\cC_f)=\frac{(q^n-1)^2}{q-1}.
\end{equation}
\end{prop}

Let $f,g \in \tilde{\mathscr{L}}_{n,q}[X]$. They are said to be $\GaL$-\textit{equivalent} (resp. $\GL$-\textit{equivalent}) if the $\Fq$-subspaces $U_f$ and $U_g$ of $\Fqn^2$ are in the same $\GaL(2,q^n)$-orbit (resp. $\GL(2,q^n)$-orbit).
In \cite[Theorem 8]{Sheekey2016}, it is shown that two codes $\cC_f$ and $\cC_g$ are equivalent if and only if $f$ and $g$ are $\GaL$-equivalent.
Moreover, note that, since the code $\cC_f$ in \eqref{Cf} is an $\Fqn$-vector space of $\tilde{\mathscr{L}}_{n,q}[X]$, its left nucleus is
$$I_L(\cC_f)= \left \{ \alpha X \in \tilde{\mathscr{L}}_{n,q}[X] \colon \alpha \in \F_{q^n} \right \}$$
and it is isomorphic to the field $\Fqn$.\\

Before investigating the right idealiser of an MRD code $\cC_f$, let us recall the following. Let $G_f=\GL(2,q^n)_{\{U_f\}}$ be the set-wise stabilizer in $\GL(2,q^n)$ of the $\Fq$-subspace $U_f$, and
$G_f^\circ=G_f\cup\{O\}$, where $O$ is the zero $2\times2$ matrix.

\begin{proposition} \cite[Lemma 4.1]{Longobardi_Marino_Trombetti_Zhou}, \cite[Proposition 3.3 and Remark 3.4]{Longobardi_Zanella2024}. \label{le:full_auto_MRD}
Let $f \in \tilde{\mathscr{{L}}}_{n,q}[X]$ be a scattered polynomial and  denote the associated MRD code by $\cC_f$.  Then $\Aut(\cC_f)$ 
consists of elements of the type
\[g \mapsto \alpha X^{q^\ell} \circ g^\sigma \circ L\]
with $L \in \tilde{\mathscr{L}}_{n,q}[X]$ invertible , $\alpha\in \F_{q^n}^*$, $\ell \in \{0,1,\ldots, n-1\}$  and $\sigma \in \Aut(\Fq)$ such that 
$$\cC_{f^{\sigma q^\ell}} \circ X^{q^\ell}\circ L= \cC_f.$$
For any $\alpha \in \Fqn^*$, $\ell \in \{0,1,\ldots,n-1\}$ and $\sigma \in \Aut(\Fq)$, there is a bijection between the set of all $L$ such that  $(\alpha x^{q^\ell}, L, \sigma) \in \Aut(\mathcal{C}_f)$ and all linear isomorphism between $U_f$ to $U_{f^{\sigma q^\ell}}$. Furthermore, 
$I_R(\cC_f)$ and $G_f^\circ$  are isomorphic fields.
\end{proposition}

\subsection{Known scattered polynomials} \label{known-scattered-polynomials}

In this section we collect the known non-equivalent, up to $\GaL(2, q^n)$-equivalence, scattered polynomials along with the stabilizers of their related subspaces and the right idealisers of the associated MRD code. Assume $d$ is a divisor of $n$; by 
$$\mathrm{N}_{q^n/q^d}(x)=x^{(q^n-1)/(q^d-1)} \quad \quad \textnormal{ and } \quad \mathrm{Tr}_{q^n/q^d}(x)=x+x^{q^d}+\ldots+x^{q^{\frac{n}{d}-1}},$$
we denote the \emph{norm} and the \textit{trace}
of an element $x\in\Fqn$ over $\F_{q^d}$, respectively. 
Moreover, we will identify any non-singular matrix $A\in\Fqn^{2\times2}$ with the map $(x,y)\mapsto A(x,y)^t$ in $\GL(2,q^n)$.
\begin{enumerate}
\item $f_s^{1,n}:=X^{q^s}$, $\gcd(s,n)=1$ (\textit{pseudoregulus}), see \cite{blokhuis_scattered_2000}.\\
  In this case, it is immediate to see that the stabilizer in $\GL(2,q^n)$ and the right idealiser of the MRD code associated $\cC_{s}^{1,n}:=\cC_{f_s^{1,n}}$ are
\begin{equation}
    G_s^{1,n}:=G_{f_s^{1,n}}=\{\diag(\alpha,\alpha^
{q^s})\colon\alpha\in\Fqn^*\} \quad \textnormal{ and }\quad  I_R(\cC^{1,n}_{s})=\{\alpha X \colon \alpha \in \F_{q^n}\},
\end{equation}
respectively, see \cite{CsMaPoZa18}.
In particular, we have that the associated MRD code with a scattered polynomial of pseudoregulus type is, up to equivalence, the only one which has both idealisers of the largest possible size. Indeed, if $f \in\tilde{\mathscr{L}}_{n,q}[X]$ is a scattered polynomial not $\GaL$-equivalent to a polynomial of
pseudoregulus type, then $G_f^\circ$ is matrix field isomorphic to a proper subfield  of $\Fqn$, see \cite[Proposition 3.6]{Longobardi_Zanella2024}.
In the remainder of the survey, we will refer to any linear set $\mathrm{PGL}$-equivalent to $L_s^{1,n}:=L_{f_{s}^{1,n}}$ as a linear set of  \textit{pseudoregulus type}.

\item $f_s^{2,n}:= X^{q^s} + \delta X^{q^{n-s}}$, $n\geq 4$, $\mathrm{N}_{q^n/q}(\delta)\notin \{0,1\}$, $\gcd(s,n)=1$,  see \cite{Lunardon_Polverino}. Here, the stabilizer in $\GL(2,q^n)$ of the associated scattered subspace is
\begin{equation}
G_{s}^{2,n}:=G_{f^{2,n}_s}=\{\diag(\alpha,\alpha^{q^s})\colon\alpha\in\F_{q^{\gcd(2,n)}}^*  \}
\end{equation}
and the right idealiser of the code $\cC_{f_s^{2,n}}$ is
\begin{equation}
I_R(\cC_{f_s^{2,n}})=
 \{\alpha X \colon\alpha\in\F_{q^{\gcd(2,n)}}  \}, 
\end{equation}
see \cite{CsMaPoZa18} and \cite{Longobardi_Zanella2024}.
Similarly to what was done above, we will refer to any linear set $\mathrm{PGL}$-equivalent to $L_{s}^{2,n}:=L_{f_s^{2,n}}$ as a linear set of  \textit{Lunardon-Polverino type} or \textit{LP} \textit{type}.
\item $f^{3,n}_s:= X^{q^s}+ \eta X^{q^{s+n/2}}$, $n \in \{6,8\}$, $\gcd(s,n/2)=1$, $\mathrm{N}_{q^n/q^{n/2}}(\eta) \notin \{0,1\}$, see \cite[Theorem 7.1 and 7.2]{CsMaPoZa18} and \cite{TiZi}. 
In this case 
\begin{equation}
    G_s^{3,n}:=G_{f^{3,n}_s}=\{\diag(\alpha,\alpha^{q^s})\colon\alpha\in\F_{q^{n/2}}^*\} \quad \textnormal{ and } \quad I_R(\cC_{f^{3,n}_s})=\{\alpha X : \alpha \in \F_{q^{n/2}}\}.
\end{equation} 
\item 
$f^{4}_\theta:=X^{q}+X^{q^{3}}+\theta X^{q^{5}} \in \tilde{\mathscr{L}}_{6,q}[X]$,  where  $\theta^2+\theta=1$ if
 $q$ is odd, \cite{CsMaZu18,MaMoZu20};  some  conditions on $\theta$, if $q$ is large enough and even, \cite[Theorem 2.9]{scatteredeven}.
In both odd and even characteristic, 
\begin{equation}
 G^{4}_{\theta}:=G_{f^4_{\theta}}=\{\diag(\alpha,\alpha^{q})\colon\alpha\in\F_{q^{2}}^*\} \quad \textnormal{ and } \quad I_R(\cC_{f^4_\theta})=\{\alpha X \colon \alpha \in \F_{q^2}\}.
 \end{equation}
\item The quadrinomial with coefficients over $\F_{q^{2t}}$, $q$ odd $t \geq 3$ and $\gcd(s,2t)=1$,
  \begin{equation}\label{quad}
      \psi_{m,h,s}:=  m \left (X^{q^s}-h^{1-q^{s(t+1)}}X^{q^{s(t+1)}} \right )+X^{q^{s(t-1)}}+h^{1-q^{s(2t-1)}}X^{q^{s(2t-1)}} \in \tilde{\mathscr{L}}_{2t,q,s}[X], 
      \end{equation}
      with $(m,h) \in \F_{q^t} \times \F_{q^{2t}}$.
This polynomial has been studied extensively by several authors and generalised in various steps which we briefly recap below. More precisely, the polynomial $\psi_{m,h,s}$ turns to be  scattered:
\begin{itemize}
    \item [$(i)$] for $m=1$ and $h \in \F_{q^t}$ with $h^2=-1$. This was introduced  for $t=3$ in \cite{Zanella_Zullo} and studied in \cite{BarZanZu}. Later, it is generalized for any $t \geq 3$ in \cite{Longobardi_Zanella2021} (see also \cite{NeSanZu}). Note that $h^2=-1$ implies that $h \in  \F_{q^{\gcd(t,2)}}$. So, if $t$ is odd, $q \equiv 1 \pmod 4$;

    \item [$(ii)$] for $m=1$ and $h \in \F_{q^{2t}} \setminus \F_{q^t}$ with $\mathrm{N}_{q^{2t}/q^t}(h)=-1$, see \cite{BarZanZu} for $(s,t)=(1,3)$ and \cite{Longobardi_Marino_Trombetti_Zhou, NeSanZu} for any integer $s$ coprime to $t$ and  $t \geq 3$;
    
    \item [$(iii)$] for $h \in \F_q$ and $m \in \F_{q^t} \setminus (\mathcal{P}^+ \cup \mathcal{P}^-)$ where
    \begin{equation}\label{Pset}
        \mathcal{P}^+=\{w^{q+1} \in \F_{q^t}  \colon w \in \ker \mathrm{Tr}_{q^{2t}/q^t} \} \quad \textnormal{ and }  \mathcal{P}^-=\{w^{q-1} \in \F_{q^t}  \colon w \in \ker \mathrm{Tr}_{q^{2t}/q^t} \},
    \end{equation}
    see \cite{SmaZanZu};
\item [$(iv)$] for $(m,h) \in \F_{q^t} \times (\F_{q^{2t}}\setminus \F_{q^t})$ with 
\begin{equation}
\begin{cases}
      \mathrm{N}_{q^{2t}/q^t}(h)=\pm 1 \text{ and } m \in \mathbb{F}_{q^t} \setminus (\mathcal{P}^+ \cup \mathcal{P}^{-}) & \textnormal{if $t$ is even or $t$ is odd  and $q \equiv 1\,\, (\mathrm{mod}\,4)$}\\ 
      \mathrm{N}_{q^{2t}/q^t}(h)=-1 \text{ and } m \in \mathcal{P}^+ & \textnormal{if $t$ is odd, $q \equiv 3\,\,(\mathrm{mod }\,4)$} \\
       \mathrm{N}_{q^{2t}/q^t}(h)=1 \textnormal{ with } h^2 \neq -1\text{ and } m \in \mathbb{F}_{q^t} \setminus (\mathcal{P}^+ \cup \mathcal{P}^{-}) & \textnormal{if $t$ is odd, $q \equiv 3\,\, (\mathrm{mod}\,4)$,}
     \end{cases}
     \end{equation}
     where $\mathcal{P}^+$ and $\mathcal{P}^-$ are as in \eqref{Pset}, see \cite{GiaGriLonTim}. 

     Finally, for the cases $(i)$-$(iii)$, we have the description of the stabilizers $G_{m,h,s}:=G_{\psi_{m,h,s}}$ of the scattered subspace $U_{\psi_{m,h,s}}$ and this is isomorphic to the multiplicative group of $\F_{q^2}$, see \cite{Longobardi_Marino_Trombetti_Zhou,Longobardi_Zanella2024, NeSanZu, SmaZanZu}. In particular, for $t$ even 
\begin{equation}
    G_{m,h,s}=\{\diag(\alpha,\alpha^q) : \alpha \in \F_{q^2}^*\} \quad \quad \textnormal{ and } \quad \quad I_R(\mathcal{C}_{\psi_{m,h,s}})=\{\alpha X \colon \alpha \in \F_{q^2}\}.
\end{equation}
\end{itemize}
\end{enumerate}

Note that some of the families of linear sets defined by the polynomials listed above, may share elements, see \cite[Proposition 3.10]{BarZanZu} and \cite[Proposition 5.5]{Zanella_Zullo}. 
	
\subsection{Standard form for scattered polynomials}
In this section, we will show that any scattered polynomial such that the right idealiser of the associated code is isomorphic to a proper extension of $\F_q$, is equivalent to one with a particular shape.
From now on, we will denote the set of all scattered polynomials $f \in \tilde{\mathscr{L}}_{n,q}[X]$
such that $G_f^\circ$ is not isomorphic to $\Fq$ by $\mathscr{S}_{n,q}$.

\begin{definition}\label{standard-form}
\textnormal{Let $h=\sum_{i=0}^{n-1}b_iX^{q^i}$ be a scattered polynomial,
\[\Delta_h=\{(i-j) \mod n\colon b_ib_j\neq0 \, \text{and} \, i \neq j \}\cup\{n\},\] and let $r_h$ be the greatest common divisor
of $\Delta_h$.
If $r_h>1$ then $h$ is in \emph{standard form}.}
\end{definition}
For instance, the scattered polynomial of LP type $h:=f_{1}^{2,n}=X^q+\delta X^{q^{n-1}}$, $n$ even,
is in standard form. Indeed, $\Delta_h=\{2,n-2,n\}$ and $r_h=2$. On the other hand if $n$ is odd, $f_{2}^{2,n}$ is not in standard form.
Moreover, if a scattered polynomial $h$ is in standard form, then it has the following shape:
$$\sum^{n/r
-1}_{
j=0}
b_jX^{q^{jr+s}}
,$$
where $r = r_h$ and $0 \leq  s < r$ with $\gcd(s,r)=1$, see \cite[Remark 4.2]{Longobardi_Zanella2024}.

%\begin{thm}\cite[Theorem 4.3]{Longobardi_Zanella2024}, \cite[Theorem 2.1]{GuLoTro} \label{t:sf}
%Let $h$ be  a scattered polynomial in $\tilde{\mathscr{L}}_{n,q}[X]$.
%The following statements are equivalent:
%\begin{itemize}
%\item [$(i)$] $h$ is in standard form;
%\item [$(ii)$] $|G_h^\circ|=q^r$, $r>1$, and all elements of $G_h$ are diagonal;
%\item [$(iii)$] the right idealizer of the code $\cC_h$ is 
%\begin{equation*}
%    I_R(\cC_h)=\{\alpha X : \alpha \in \F_{q^r}\}.
%\end{equation*}
%\end{itemize}
%If the conditions $(i)$, $(ii)$ and $(iii)$ above hold, then $r=r_h$ and
%\begin{equation}\label{eq:sfgf}
%G_h^\circ=\biggl \{\begin{pmatrix}
 %   \alpha & 0\\
 %   0 & \alpha^{q^s}
%\end{pmatrix} \colon \alpha \in \F_{q^{r}} \biggr\}
%\end{equation}
%where $1 \leq s < r$ and $\gcd(s,r)=1$.
%\end{thm}
%\end{comment}

Any scattered polynomial in standard form is bijective, see \cite[Theorem 4.5]{Longobardi_Zanella2024}. Indeed, if  $h$ is a scattered polynomial in standard form, it induces an $\F_{q^r}$-semilinear map $a \in \Fqn \longmapsto h(a) \in \Fqn$ where $r=r_h > 1$. This implies that its kernel is an $\F_{q^r }$-subspace of $\F_{q^n}$. 
Now, if $0 \neq  \bar{x} \in\ker h$, then 
$\dim_{\F_q} (U_h \cap \langle (\bar{x},0)\rangle_{\F_{q^n}}) >1$, against the assumption that $h$ is scattered. Hence, the kernel $\ker h$ must be the null space.

\begin{thm}\cite[Corollary 4.7]{Longobardi_Zanella2024}
Let $f$ be  a scattered polynomial in $\mathscr{S}_{n,q}$. 
Then $f$ is $\GL$-equivalent to a polynomial $h$ in standard form.
\end{thm}

We can refer to the scattered polynomial $h$ above as \textit{a standard form of} $f$.
Actually, a standard
form for a scattered polynomial $f$ is \textit{essentially unique}. More precisely, the following holds.

\begin{proposition}\cite[Proposition 4.8]{Longobardi_Zanella2024}
Let $h_1$ and $h_2$ be scattered polynomials in standard form belonging to $\tilde{\mathscr{L}}_{n,q}[X]$.  If they are $\GL$-equivalent to a scattered 
polynomial $f \in \tilde{\mathscr{L}}_{n,q}[X]$, then there exist $a,b\in\Fqn^*$ such that  
\begin{equation*}
    bh_2 = h_1 \circ aX \quad \quad \textit{ or }\quad \quad  h_1
\circ a h_2 = bX.
\end{equation*}
\end{proposition}

Moreover, since any polynomial that is $\GaL$-equivalent to $f$ is $\GL$-equivalent to $f^\rho$ for
some automorphism $\rho$ of $\Fqn$,  we can state the result.

\begin{thm}  Let $f_1$,$f_2$ be scattered polynomials in $\mathscr{S}_{n,q}$ having $h_1$ and $h_2$ as standard forms. These are $\GaL$-equivalent if and only if there exist $a,b\in\Fqn^*$ and an
automorphism $\rho$ of $\Fqn$, such that \begin{equation*}
    bh_2 = h_1^\rho \circ aX \quad \quad \textit{ or }\quad \quad  h_1^\rho
\circ a h_2 = bX
\end{equation*}
\end{thm}

As a consequence on the codes associated with two scattered polynomials in $\mathscr{S}_{n,q}$, we can state the following.

\begin{cor} \cite[Theorem 2.2]{GuLoTro} Let $\cC_i=\cC_{f_i}$ be maximum rank distance codes  as in \eqref{Cf} such that $I_R(\cC_i)$, are not isomorphic to $\Fq$, $i = 1, 2$.
Then, $\cC_1$ and $\cC_2$ are equivalent if and only if there exists $a,b\in\Fqn^*$ and an
automorphism $\rho$ of $\Fqn$, such that \begin{equation*}
    bh_2 = h_1^\rho \circ aX \quad \quad \textit{ or }\quad \quad  h_1^\rho
\circ a h_2 = bX,
\end{equation*}
where $h_i$ is a standard form of $f_i$, $i=1,2$.
\end{cor}

Note that for $t$ even, the scattered polynomial $\psi_{m,h,s}$ is in standard form. On the other hand, if $t$ is odd, $\psi_{m,h,s}$ is not. In \cite[Example 4.10 and 4.11]{Longobardi_Zanella2024}, a standard form for $\psi_{m,h,s}$ with $h \in \F_q$ and for $\psi_{1,h,s} \in \mathscr{L}_{6,q,s}[X]$  is computed.\\
This allowed the study of the equivalence issue among codes of the type $\cC_{\psi_{m,h,s}} \subset  \tilde{\mathscr{L}}_{2t,q,s}[X]$ where $t \in \{3,4\}$. Indeed, in \cite{NeSanZu} Neri \textit{et al.} determined 
the equivalence classes of the codes  in this  family and
provided the exact number of inequivalent ones in it for $t \geq 5$. In \cite{GuLoTro}, using the notion of standard form of a scattered polynomials, the authors completed this study removing this restriction.

\section{\texorpdfstring{The $\GaL$}{TEXT}-class of a linear set and the code \texorpdfstring{$\cC_{\widehat{f}}$}{TEXT}}

As seen in Subsection \ref{linearsets}, an $\F_q$-linear set can be defined by different $\F_q$-subspaces. These can be  
in the same orbit under the action of the group $\mathrm{\Gamma L}(r,q^n)$ or not. For instance, if $s \ne 1$ the $\F_q$-subspaces $U=\{ (x,x^q):x \in \F_{q^n} \}$ and $W=\{ (x,x^{q^s}): x \in \F_{q^n}\}$ are not $\mathrm{\Gamma L}$-equivalent but define the same linear set. Indeed,
$$\{ \langle (1,x^{q-1}) \rangle_{\F_{q^n}}: x \in \F_{q^n}^*\}=\{ \langle (1,x) \rangle_{\F_{q^n}}: x \in \F_{q^n}^*, \mathrm{N}_{q^n/q}(x)=1\}=\{ \langle (1,x^{q^s-1}) \rangle_{\F_{q^n}}: x \in \F_{q^n}^*\},$$
see  \cite{csajbok_zanella}.  This leads to the introduction of the notion of $\mathcal{Z}(\GaL)$-class and $\GaL$-class of a linear set of rank $n$ in $\PG(1,q^n)$, \cite{classes} . Precisely,
\begin{definition}\cite[Definition 2.4]{classes}
Let $L=L_U$ be an $\F_q$-linear set of $\PG(1, q^n)$ of rank $n$ with maximum field of linearity $\F_q$, i.e., $L$ is not an $\F_{q^v}$-linear set
for any $v > 1$. Then $L_U$ has $\mathcal{Z}(\GaL)$-\textit{class} $z$, if $z$ is the largest integer
such that there exist $\F_q$-subspaces $U_1, U_2,\ldots, U_z$ of $\F_{q^n}^2$ with 
\begin{enumerate}
    \item $L=L_{U_i}$ for
$i \in \{1, 2, \ldots, z\}$, and
\item there is no $\lambda 
 \in \F_{q^n}^*$ such that $U_i= \lambda U_j$  for each
$i \neq  j$ and  $i, j \in \{1, 2,\ldots, z\}$.
\end{enumerate} 
\end{definition}
 Similarly, the notion of $\GaL$-class of an $\F_q$-linear set was defined.

\begin{definition} \cite[Definition 2.5]{classes}. Let $L=L_U$ be an $\F_q$-linear set of $\PG(1, q^n)$ of rank $n$ with maximum field of linearity $\F_q$. Then $L$ is of $\GaL$-\textit{class} $c$, if $c$ is the largest integer
such that there exist $\F_q$-subspaces $U_1,U_2,\ldots,U_c$ of $\F_{q^n}^2$ with
\begin{enumerate}
    \item  $L=L_{U_i}$ for $i\in \{1, 2, \ldots, c\}$, and
    \item there is no $\kappa  \in  \GaL(2, q^n)$ such that $U_i = U_j^\kappa$ for each $i \neq j$, $i, j \in \{1, 2,\ldots, c\}$.
\end{enumerate} 
\end{definition}

In the following, we will denote the $\mathcal{Z}(\GaL)$-class and the $\GaL$-class of a linear set $L$ by $c_{\mathcal{Z}(\GaL)}(L)$ and $c_{\GaL}(L)$, respectively. These integers are invariant under equivalence, see \cite[Proposition 2.6]{classes}, and it is clear that $c_{\GaL}(L)  \leq c_{\mathcal{Z}(\GaL)}(L)$. The linear set $L$ is called $simple$ if its  $\GaL$-class is  one.\\
For instance, in \cite{csajbok_zanella} the authors proved that an $\F_q$-linear set of pseudoregulus type $L^{1,n}_s:=L_{f^{1,n}_s}$ has  $\GaL$-class $\phi(n)/2$, where $\phi$ is the totient function, hence they are non-simple for $n = 5$ and $n > 6$. The existence of other non-simple linear sets of rank $n$ in $\PG(1,q^n)$ is investigated in \cite{classes}. 
In order to understand the idea behind these results, let us recall the following.

\begin{thm} \cite{BaGiuMaPo}, \cite[Lemma 3.1]{classes}. Let 
$L_f$ be an $\F_q$-linear set in $\PG(1,q^n)$ of rank $n$, 
$f \in  
\mathscr{\tilde{L}}_{n,q}[X]$, and let 
$\hat{f}$ denote its adjoint. Then, the linear sets defined by  $f$ and 
$\hat{f}$ coincide, i.e., 
$L_f=L_{\hat{f}}$. In particular, if 
$f$ is scattered,  then $\hat{f}$ is also scattered.
\end{thm}

Therefore, if $L=L_f$ is an $\Fq$-linear set such that the $U_f$ and $U_{\hat{f}}$ are not in the same $\GaL(2,q^n)$ orbit,  $c_{\GaL}(L)>1$ , see \cite[Section 3 and 4]{classes}.  
In \cite{Pepe}, Pepe studied the link of subspaces defining the same linear set, characterising those of rank $n$. This result has the following consequence for the maximum scattered ones of the projective line.

\begin{thm}\cite[Theorem 3.4]{GriGuLoTro}\label{thm:Gamma_L_class}
Let $L_f$ be a maximum scattered linear set of $\PG(1,q^n)$ with $f \in \tilde{\mathscr{L}}_{n,q}[X]$. If $W$ is an $n$-dimensional $\F_q$-subspace of $\F_{q^n}^2$ such that $L_{W}=L_f$, then one of the following possibilities occurs:
\begin{enumerate}
 \item $L_f$ is equivalent to a linear set of \textit{pseudoregulus type}
\begin{equation*}
 L^{1,n}_{s}=\{\langle (x, x^{q^s}) \rangle_{\F_{q^n}} \, \colon \, x \in \F^*_{q^n} \};
 \end{equation*}
\item $W= \lambda U_f$ for some $\lambda \in \F^*_{q^n}$;
\item $W= \lambda U_{\hat{f}}$ for some $\lambda \in \F_{q^n}^*$, where $\hat{f}$ is the adjoint of $f$.
\end{enumerate}  
In particular, if $L_f$ is not of pseudoregulus type  $c_{\GaL}(L_f) \leq 2$.
\end{thm}

Note that even if the adjoint polynomial $\hat{f}$ of a scattered polynomial $f$ is scattered as well,  the rank-metric code $\cC_{\hat{f}}$ associated with $\hat{f}$ may not be equal nor equivalent to the adjoint code $\widehat{\cC_{f}}$ of $\cC_f$. 
Clearly, if $f_s^{1,n}=X^{q^s}$, with $\gcd(n,s)=1$, and considered $\cC_{s}^{1,n}:=\cC_{f_s^{1,n}}$, it is straightforward to check that
\begin{equation}\label{pseudo-code}
\cC_{\widehat{f^{1,n}_s}}=\cC^{1,n}_{n-s}=\cC_s^{1,n} \circ X^{q^{n-s}}=\widehat{\cC_{s}^{1,n}}.
\end{equation}

In the remainder of this section, we will show some new results for $\cC_f$ and $\cC_{\hat{f}}$ where $f \in \tilde{\mathscr{L}}_{n,q}[X]$ is a scattered polynomial.

\begin{lem} \label{eq-adjoint}
Let $f,g \in \tilde{\mathscr{L}}_{n,q}[X]$ be scattered polynomials. If $\cC_f$ and $\cC_g$ are equivalent, then  $\cC_{\hat{f}}$ and  $\cC_{\hat{g}}$ are equivalent.
\end{lem}
\begin{proof}
Since $\cC_f$ and $\cC_g$ are equivalent, we have that $f$ and $g$ are $\GaL$-equivalent and hence  $f^\sigma$ is $\GL$-equivalent to $g$ for some automorphism $\sigma \in \Aut(\F_{q^n})$, \cite[Theorem 8]{Sheekey2016}.  Then, there exist $a,b,c,d\in\F_{q^n}$ with $ad-bc\ne 0$, such that for any $x \in \F_{q^n}$
\[
\begin{pmatrix}
a&b\\
c&d\\
\end{pmatrix}
\begin{pmatrix}
x\\
f^\sigma(x)
\end{pmatrix}=
\begin{pmatrix}
y\\
g(y)
\end{pmatrix}
\]
where $y \in \F_{q^n}$.
It is a straightforward computation to show that 
\[
\begin{pmatrix}
-d&b\\
c&-a\\
\end{pmatrix}
\begin{pmatrix}
x\\
\hat g(x)
\end{pmatrix}=
\begin{pmatrix}
y\\
\hat f^\sigma(y)
\end{pmatrix}.
\]
Hence, $\cC_{\hat{g}}$ is equivalent to $\cC_{\hat{f}^\sigma}$. Since $\cC_{\hat{f}^\sigma} \cong \cC_{\hat{f}}$, we get the result.
\end{proof}

%\begin{example} 
%\textnormal{Let $\psi:=\psi_{m,h,s} \in \tilde{\mathscr{L}}_{2t,q}$ be the scattered polynomial in \eqref{quad} with $m=h=s=1$, $t \geq 3$ and $q$ odd. This comes down to
%$$\psi= X^q+X^{q^{t-1}}-X^{q^{t+1}}+X^{q^{2t-1}}.$$
%Let suppose that $\cC_{\hat{\psi}}$ is equivalent to $ \hat{\cC}_{\psi}$. Then, by \cite[Propositions 4.1 and 4.2, Theorem 5.4 and Corollary 5.6]{Lunardon_Trombetti_Zhou},   we get
%\[\F_{q^2} \cong I_R(\cC_{\hat{\psi}}) \cong %I_R(\widehat{\cC_{\psi}})\cong \widehat{I_L(\cC_{\psi_t})}\cong \F_{q^{2t}},\]
%a contradiction.}
%\end{example}

\begin{prop}
If $f \in\tilde{\mathscr{L}}_{n,q}[X]$ is a scattered polynomial not $\GaL$-equivalent to a polynomial of
pseudoregulus type, then the adjoint code $\widehat{\cC_f}$ is not equivalent to $\cC_{\hat{f}}$.
\end{prop}
\begin{proof}
    Suppose that $\cC_{\hat{f}}$ is equivalent to $ \widehat{\cC_{f}}$. Firstly,  $\cC_{\hat{f}}$ cannot be equivalent   to a code $\cC_{f^{1,n}_s}$. Indeed, if this is the case, by \eqref{pseudo-code} and Lemma \ref{eq-adjoint} we have that $f$ is equivalent to a polynomial of pseudoregulus type, against our hypothesis. Then, by \cite[Proposition 3.6]{Longobardi_Zanella2024}, $I_R(\cC_{\hat{f}})$ is a proper subfield $\F_{q^m}$ of $\F_{q^n}$ with $m \mid n$. By \cite[Propositions 4.1 and 4.2, Theorem 5.4 and Corollary 5.6]{Lunardon_Trombetti_Zhou},   we get
\[\F_{q^m} \cong I_R(\cC_{\hat{f}}) \cong I_R(\widehat{\cC_{f}})\cong \widehat{I_L(\cC_{f})}\cong \F_{q^{n}},\]
a contradiction.
\end{proof}

Let $h$ be  a scattered polynomial, it is straightforward to see that $\Delta_h=\Delta_{\hat{h}}$ and hence $r_h=r_{\hat{h}}$ (cf. Definition \ref{standard-form} ). In particular, we have that if $h$ is in standard form, then its adjoint polynomial $\hat{h}$ is in standard form as well.
 \begin{proposition}
Let $f$ be a scattered polynomial in $\tilde{\mathscr{L}}_{n,q}[X]$. Then, $I_R(\cC_f)$ is a field isomorphic to $I_R(\cC_{\hat{f}})$.
\end{proposition}
\begin{proof}
Let $f$ be a scattered polynomial and first assume that it belongs to  $\mathscr{S}_{n,q}$. Since $f$ is $\mathrm{GL}$-equivalent to a polynomial $h$ in standard form, the code $\cC_f$ is equivalent to $\cC_{h}$. By \cite[Theorem 2.1]{GuLoTro} and Lemma \ref{eq-adjoint}, $\cC_{\hat{f}}$ is equivalent to $\cC_{\hat{h}}$, $r=r_h=r_{\hat{h}}$ and
$$I_R(\cC_h)= I_R(\cC_{\hat{h}})= \{\alpha X : \alpha \in \F_{q^r}\}.$$
By \cite[Proposition 3.1]{Lunardon_Trombetti_Zhou}, $I_{R}(\cC_f)  \cong \F_{q^r} \cong I_{R}(\cC_{\widehat{f}})$.
Now, let us assume that $f \not \in \mathscr{S}_{n,q}$. Then  $\hat{f}$ cannot belong to $\mathscr{S}_{n,q}$. Indeed, if this is the case then $\hat{f}$ is $\mathrm{GL}$-equivalent to a certain scattered polynomial $g$ in standard form. This implies that the field $I_R(\cC_f)$ is isomorphic to $I_R(\cC_{\hat{g}})=I_R(\cC_g)$ and this is isomorphic to a field $\F_{q^r}$ with $r=r_g>1$, a contradiction. Hence, the right idealisers $I_R(\cC_f)$ and $I_R(\cC_{\hat{f}})$ are both fields isomorphic to $\F_q$. This concludes the proof.
\end{proof}

\section{Classification and characterization results}

In this section, we will retrace the state of the art regarding the classification of maximum scattered linear sets of the projective line $\PG(1,q^n)$ and briefly mention some results concerning geometrical characterizations of those arising from the scattered polynomials listed in Subsection \ref{known-scattered-polynomials}.\\
Firstly, a maximum scattered $\F_2$-linear set in $\PG(1,2^n)$ is equivalent to a translation hyperoval in $\PG(2,2^n)$. Due to the famous classification result of translation hyperovals in Desarguesian planes by Payne \cite{payne_complete_1971}, we have the following result.
	\begin{lem}\label{le:Payne}
		Let $f$ be a scattered polynomial in $\mathscr{L}_{n,2}[X]$, then $f$ is equivalent to $X^{2^s}$ where  $\gcd(s,n)=1$.
	\end{lem}

Hence, every maximum scattered $\F_2$-linear set in $\PG(1,2^n)$ is of pseudoregulus type.\\
As seen in Subsection \ref{linearsets},  (canonical) $\Fq$-subgeometries of a finite projective space are examples of linear sets. An $\F_q$-subgeometry of $\PG(r-1,q^n)$ is a set projectively equivalent 
to the linear set of rank $r+1$ associated with the $\F_q$-subspace $\F_{q}^{r+1}$, i.e. the set of points with rational coordinates over $\Fq$.
These sets play a crucial role in the study of the characterization and the classification of maximum linear sets of the projective line. Their relevance  lies in the following result. 
\begin{thm} \cite{Limbos} and \cite[Theorem 2]{LuPo2004}. \label{t:lupo}
If $L=L_U$ is an $\Fq$-linear set of rank $\nu$ in $\Lambda=\PG(d,q^n)$
such that $\langle L\rangle=\Lambda$, then either $L$ is an $\F_q$-subgeometry of $\Lambda$ or there exists
a projective space $\PG(\nu-1,q^n)\supseteq \Lambda$, an $\Fq$-canonical subgeometry $\Sigma$
of $\PG(\nu-1,q^n)$, and a complement $\Gamma$ of $\Lambda$ in $\PG(\nu-1,q^n)$, such that $\Gamma\cap\Sigma=\emptyset$, and $L$ is the projection of $\Sigma$
from $\Gamma$ onto $\Lambda$. 
Conversely, any such projection is a linear set.
\end{thm}

The subspaces $\Gamma$ and $\Lambda$ are also called {\it vertex} (or \textit{center}) and {\it axis} of the projection, respectively. 
We will call any such pair $(\Gamma,\Sigma)$ a \emph{projecting configuration} for the linear set $L_U$.  
Since the axis is immaterial, the properties of $L_U$ can be described in terms of such projecting configuration.
The corresponding projection map is denoted by $\wp_\Gamma$.\\
Clearly, the maximum scattered $\F_q$-linear sets of $\PG(1, q^2)$ are the Baer sublines.  Since subgeometries are simple linear sets, these have $\mathcal{Z}(\GaL)$-class one, cf. \cite[Theorem 2.6]{LavrauwVanderVoorde} and \cite[Section 25.5]{GGG}.

According to Theorem~\ref{t:lupo}, for $n \geq 3$, a maximum $\F_q$-linear set $L_U$ of the projective line can be seen embedded in $\PG(n-1,q^n)$, and it has as a vertex 
an $(n-3)$-dimensional projective subspace $\Gamma$ projecting a subgeometry $\Sigma$, 
$\Sigma\cap\Gamma=\emptyset$, on a line $\ell$ such that $\Gamma\cap\ell=\emptyset$.\\
The mutual position of $\Gamma$ and $\Sigma$ and the behavior of
the vertex $\Gamma$ under the action of the pointwise stabilizer 
$\mathbb G=\PGaL(n,q^n)_\Sigma$ of $\Sigma$ in $\PGaL(n,q^n)$
 have been the subjects of investigation in various papers \cite{classes,csajbok_zanella,pseudoregulus-type,GriGuLoTro, LavrauwVanderVoorde, LiLoZa, MZ2019, Zanella_Zullo}. This has allowed the classification of maximum scattered linear sets of the line for small values of the rank $n$, and the characterization of linear sets known so far   via the geometric properties of the projecting configuration.\\
 In \cite{csajbok_zanella}  it was proved that 
for some linear sets of $\PG(1,q^n)$ the projecting configuration is not unique up to collineations of $\PG(n-1,q^n)$. Indeed, the $\GaL$-class of a linear set may be rephrased via a group action property on its projection vertices, see \cite[Section 5.2]{classes}, \cite[Theorem 6 and 7]{csajbok_zanella} and \cite[Theorem 3.4]{Zanella_Zullo}. More precisely, let $\Aut(\Sigma)$ be the \textit{automorphism group} of a subgeometry $\Sigma$ in $\PG(n-1,q^n)$, i.e. the group made up of all the maps in $\PGaL(n,q^n)$ which commute with a generator $\sigma$ of $\mathbb{G}$. Then, the following result holds. 

 \begin{thm} \cite[Theorem 3.4]{Zanella_Zullo}  \label{geoclass} 
 Let $\Sigma$ be a canonical subgeometry and  $\ell$  be a line of $\PG(n-1,q^n)$. The  $\GaL$-class of a linear set $L=L_U \subset \ell$ is the number of orbits, under the action of $\Aut(\Sigma)$ of $(n - 3)$-subspaces $\Gamma$ of $\PG(n-1,q^n)$ disjoint from $\Sigma$ and from the line $\ell$ such that $\wp_{\Gamma}(\Sigma)$ is equivalent to $L$.
\end{thm}

Assume that $(\Gamma,\Sigma)$ is a projecting configuration in $\PG(n-1,q^n)$ for a linear set $L_U$ and let $\sigma$ be a generator of $\mathbb G=\PGaL(n,q^n)_\Sigma$.
The \emph{rank} of a point $P\in\PG(n-1,q^n)$ with respect to 
the $\Fq$-subgeometry $\Sigma$ is
	\[ \rk P=1+\dim\langle P,P^{\sigma},P^{\sigma^2},\ldots,P^{\sigma^{n-1}}\rangle, \]
 and it is straightforward to see that the following assertions are equivalent:
\begin{itemize}
\item [$(i)$] $\wp_\Gamma(\Sigma)$ is scattered;
\item [$(ii)$] the restriction of $\wp_\Gamma$ to $\Sigma$ is injective;
\item[$(iii)$] every point of the projecting vertex $\Gamma$ has rank greater than two.
\end{itemize}

In the following, we will state a characterization result for the linear sets of pseudoregulus and LP type by means of some geometric properties of their vertices and classification results for a maximum linear set of the projective line $\PG(1,q^n)$ with $q>2$ and $n \in \{3,4,5\}$.

\begin{thm} \cite[Theorem 2.3]{pseudoregulus-type}  \label{pseudchar}
Let $\Sigma$ be a canonical subgeometry of $\PG(n-1, q^n)$, $q > 2$, $n \geq 3$. Assume that $\Gamma$ and $\ell$ are $(n-3)$-subspace and a line of $\PG(n-1, q^n)$, respectively, such that $\Sigma\cap\Gamma=\ell\cap\Gamma= \emptyset$. Then the following assertions are equivalent:
\begin{enumerate} 
    \item  the set $\wp_{\Gamma}(\Sigma)$ is a scattered $\mathbb{F}_q$-linear set of pseudoregulus type;
    \item  there exists a generator $\sigma$ of the group $\mathbb{G}$ such that $\dim(\Gamma\cap\Gamma^{\sigma})=n-4$; 
     \item there exists a point $P$ and a generator $\sigma$ of $\mathbb{G}$ such that $\rk P=n$, and
$$\Gamma=\langle P ,P^{\sigma},\ldots, P^{\sigma^{n-3}}\rangle.$$
\end{enumerate}

\end{thm}

As consequence of the previous result, we obtain that a maximum  
$\F_q$-linear set of $\PG(1,q^3)$ has a rank three point of $\PG(2,q^3)$ as its projecting vertex. Hence, Condition 3. in the theorem above is always satisfied.

\begin{thm}
Any maximum scattered $\F_q$-linear set in $\PG(1,q^3)$ is of pseudoregulus type.
\end{thm}

This has been already obtained using that the stabilizer of a
subgeometry $\Sigma$ of $\PG(2, q^3)$ is transitive on the set of those points of $\PG(2,q^3) \setminus \Sigma$
which are incident with a line of $\Sigma$ and on the set of points of $\PG(2,q^3)$ not incident with any line
of $\Sigma$, see  \cite[Example 5.1, Section 5.2 and Remark 5.6]{classes} and \cite{LaVan}.\\
Since the $\F_q$-linear sets of
rank $4$ in $\PG(1, q^4)$ with maximum field of linearity $\F_q$ are simple \cite[Theorem 4.5]{classes}, in \cite{PG(1q4)} Csajb\'{o}k and Zanella proved the following classification result. 

\begin{thm}\cite[Theorem 3.4]{PG(1q4)}
 Any maximum scattered $\F_q$-linear set in $\PG(1,q^4)$ is of pseudoregulus type or LP type.
\end{thm}
\noindent Similarly to what was done for pseudoregulus type linear sets, in \cite{Zanella_Zullo} the authors characterized linear sets of LP type and introduced the notion of intersection number of a subspace with respect to a collineation fixing pointwise a canonical subgeometry. More precisely, let $\Gamma$ be a non-empty $k$-dimensional subspace of $\PG(n-1,q^n)$, the \textit{intersection number of $\Gamma$ with respect to $\sigma$}, denoted by $\ints$, is defined as the least positive  integer $\gamma$ satisfying 
\begin{equation*}
    \dim(\Gamma \cap \Gamma^\sigma \cap \ldots \cap \Gamma^{\sigma^\gamma}) > k -2 \gamma.
\end{equation*}

Clearly, the intersection number $\ints$ is invariant under the action of $\Aut(\Sigma)$. By Theorem \ref{pseudchar}, a maximum scattered linear set is of pseudoregulus type if and only if $\ints=1$ for some generator $\sigma$ of the group $\mathbb{G}$.  In \cite{GriGuLoTro}, by a slight  strengthening of \cite[Theorem 3.5]{Zanella_Zullo}, the following result that characterizes maximum scattered linear sets of LP type is stated.

\begin{thm} \cite[Theorem 3.7]{GriGuLoTro} \label{LPcharct} Let $L$ be a maximum scattered linear set in a line $\ell$ of $\PG(n-1,q^n)$, $n \geq 4$ and $q>2$. Then $L$ is a linear set of LP type if and only if for each $(n - 3)$-subspace $\Gamma$ of $\PG(n - 1, q^n)$ such that $L = \wp_{\Gamma}(\Sigma)$, the following holds:
\begin{itemize}
\item [$i)$] there exists a generator $\sigma$ of $\mathbb{G}$ such that $\ints= 2$;
\item [$(ii)$] if $P$ is the unique point of $\PG(n -1, q^n)$ such that
\begin{equation*}
    \Gamma = \langle P, P^\sigma,\ldots, P^{\sigma^{n-4}},Q \rangle,
\end{equation*}
then the line $\langle P^{\sigma^{n-1}}, P^{\sigma^{n-3}}\rangle$ meets $\Gamma$.
\end{itemize}
\end{thm}

Also in \cite{LiLoZa,MZ2019}, in order to classify maximum scattered $\F_q$-linear sets of $\PG(1,q^5)$, the geometric properties of their projecting vertex in $\PG(4,q^5)$ are investigated.
In this case, every maximum scattered linear set is the projection of an $\Fq$-subgeometry $\Sigma$ of $\PG(4,q^5)$ from a plane $\Gamma$ external to the secant variety to $\Sigma$ (see \cite{nonlinear}, \cite{lavrauw} and \cite{lavrauw2}  for the notion of secant variety in this context).\\
Let $\sigma$ be a generator of $\mathbb G=\PGaL(5,q^5)_\Sigma$, and let $A=\Gamma\cap\Gamma^{\sigma^4}$,
$B=\Gamma\cap\Gamma^{\sigma^3}$.
By Theorem \ref{pseudchar}, if $A$ and $B$ are not both points, then the projected linear set is of pseudoregulus type. 
If they are both points and at least one of them has rank $5$, then the associated maximum scattered linear set must be of LP type, see \cite[Theorem 6.1]{LiLoZa}.
So, if a maximum scattered linear set of a new type exists, it must be such that $\rk A=\rk B=4$. In \cite[Section 6]{LiLoZa}, the authors derived that any maximum scattered $\F_q$-linear set in $\PG(1,q^5)$ is, up to equivalence in $\PGaL(2,q^5)$, one of the following:
\begin{enumerate}
\item [(C1)] pseudoregulus type $L_{s}^{1,5}$;
\item [(C2)] LP type $L_{s,\delta}^{2,5}$;
\item [(C3)]
$L_{\eta,\rho}=\{\langle(\eta(x^q-x)+\mathrm{Tr}_{q^5/q}(\rho x), x^q-x^{q^4})\rangle_{\F_{q^5}}\colon x\in\F_{q^5}^*\}$ where
$\eta,\rho \in \F_{q^5}$, $\eta \neq 0$ and $\mathrm{Tr}_{q^5/q}(\eta)=0\neq\mathrm{Tr}_{q^5/q}(\rho)$;
\item [(C4)]$L_\xi=\{\langle(x,\xi(x^{q}+x^{q^3})+x^{q^2}+x^{q^4})\rangle_{\F_{^5}}\colon x\in\F_{q^5}^*\}$ with $\mathrm{N}_{q^5/q}(\xi)=1$.
\end{enumerate}
The classes of sets of types (C3) and (C4) might be empty. Indeed, as an exhaustive analysis by computer showed,  no  maximum scattered linear sets exist with $\rk A= \rk B=4$ for $q \leq 25$.\\

Finally, in the spirit of \cite{csajbok_zanella, Zanella_Zullo}, in \cite{GriGuLoTro} the authors characterized the other families of maximum scattered linear sets of $\PG(1,q^n)$ arising from the scattered polynomials listed in Subsection \ref{known-scattered-polynomials}. This was achieved by studying specific subspaces that intersect the projecting vertices, yielding characterization results similar to Theorems \ref{pseudchar} and \ref{LPcharct} mentioned above.

\section{Translation planes arising from scattered polynomials}

In this section, we will explore a different property of scattered polynomials; precisely, a connection they have with  translation planes. In order to ensure that the survey is as self-contained as possible, we will first introduce the necessary notions and definitions.\\
A \emph{(partial) planar spread} of a finite $2n$-dimensional vector space $V$ is a (partial) $n$-spread $\mathcal{B}$ of 
 $V$. As seen in Subsection \ref{linearsets}, the \emph{Desarguesian spread} of the $2n$-dimensional $\Fq$-vector space $\F_{q^n}^2$
\[
  \mathcal{D}=\left\{\langle \textbf{v}\rangle_{\F_{q^n}}\colon \textbf{v}\in \left (\F_{q^n}^2 \right )^*\right\}
\]
is a planar spread of $\F_{q^n}^2$.\\
A finite (right) \emph{quasifield} is an ordered triple $(Q,+,\circ)$, where $Q$ is a  finite set
\begin{itemize}
  \item [$(i)$]  $(Q,+)$ is an abelian group, 
  \item [$(ii)$] $(Q^*,\circ)$ is a loop,
  \item [$(iii)$]  $\forall x,y,m\in Q$, \, $(x+y)\circ m=x\circ m+y\circ m$,
\item [$(iv)$] $ \forall a,b,c\in Q$ with $a\neq b$,  $\exists !\, x \in Q$ such that $x\circ a=x\circ b+c$.
\end{itemize}
The \emph{kernel} of the quasifield  $(Q,+,\circ)$ is
\begin{gather*}
  K(Q)=\{k\in Q\colon k\circ(x+y)=k\circ x+k\circ y\mbox{ and }
  k\circ(x\circ y)=(k\circ x)\circ y\\ \mbox{ for any }x,y\in Q\}.
\end{gather*}
The kernel of $Q$ is a finite field and $Q$ is a left vector space over $K(Q)$, see
\cite{An54} and \cite[Chapter 1]{Kn95}.
A  quasifield satisfying the right distributive property is called a \emph{semifield}; an associative quasifield is called a
\emph{nearfield}.
Let us define $B_\infty=\{0\}\times Q$, and
\[
  B_m=\{(x,x\circ m)\colon x\in Q\}\mbox{ for }m\in Q.
\]
Then, $\cB(Q)=\{B_m\colon m\in Q\cup\{\infty\}\}$ is a planar spread of the $K(Q)$-vector space
$Q^2$.
Conversely, for any planar spread $\cB$ a quasifield $Q$ exists such that $\cB=\cB(Q)$.

Let $\cB$ be a planar spread of  $V$.
Other planar spreads can be constructed by using the following technique of \emph{net replacement}.
If $\cB'$ and $\cB''$ are distinct partial planar spreads of $V$ such that
$\cB'\subseteq\cB$, and
\[
  \bigcup_{B'\in\cB'}B'=\bigcup_{B''\in\cB''}B'',
\]
then $(\cB\setminus\cB')\cup\cB''$ is a planar spread of $V$.
%Such $\cB'$ is called \emph{replaceable (or derivable) net}.
If $|\cB'|=|\cB''|=(q^n-1)/(q-1)$,
the partial spreads $\cB'$ and $\cB''$ are called \emph{hyper-reguli} \cite[Chapter 19]{BiJhJo07}.
Such hyper-reguli are also called \emph{replacement set} of each other.

The \emph{translation plane $\ptra(\cB)$ associated with the planar spread} $\cB$ of  a finite vector space $V$
is the plane whose points are the elements of $V$  and whose
lines are the cosets of the elements of $\cB$ in the group $(V,+)$.
Let $Q$ be a quasifield, then the \emph{translation plane $\ptra(Q)$ associated with $Q$}  is the translation plane $\ptra(\cB(Q))$ associated with $\cB(Q)$. Any line 
is represented by the equations of type $X=b$ and $Y=X\circ m+b$ with $m,b\in Q$.
It is easy to see that if $\ptra(Q)$ and $\ptra(Q')$ are isomorphic
translation planes associated with the quasifields $Q$ and $Q'$ respectively, then $K(Q)$ and $K(Q')$ are isomorphic fields.
Hence the kernel of $Q$ is a geometric invariant of the plane $\ptra(Q)$,
also called the \emph{kernel of the translation plane $\ptra(Q)$}, see \cite{BiJhJo07,Kn95} and all the references therein.
A class of the most investigated translation planes in the litereature is that of Andr\'{e} $q$-plane. More precisely, an \emph{Andr\'e $q$-plane} is associated with the planar spread of $\F_{q^n}^2$ obtained from
$\mathcal{D}$ by replacing each partial spread of type
\[
  \cB'_{\xi}=\{\langle(1,m)\rangle_{\Fqn}\colon m\in \F_{q^n}  \text{ with } \mathrm{N}_{q^n/q}(m)=\xi\}, \qquad \xi\in\Fq^*
\]
with
\[
  \cB''_{\mu(\xi)}=\left\{\{ (x,x^{\mu(\xi)}m) \colon x\in\Fqn\}\colon \mathrm{N}_{q^n/q}(m)=\xi\right\},
\]
where $\mu:\Fq^*\longrightarrow \mathrm{Gal}(\Fqn \vert \F_q)$ is a map.
Any $\cB'_{\xi}$ is called \emph{Andr\'e $q$-net} and if $\mu$ is the map $\xi \mapsto \xi^{q^s}$,
the partial spread $\cB''_{\mu(\xi)}$ is called an \emph{Andr\'e $q^s$-replacement}
\cite[Definition 16.1 and 16.3]{BiJhJo07}.

In \cite{Lunardon_Polverino}, Lunardon and Polverino showed that the pseudoregulus linear set,  contained in $\mathcal{D}$,  is a hyperregulus of $\F_{q^n}^2$ and it leads to a net replacement.
In \cite{Casarino_Longobardi_Zanella}, the authors showed that this argument can be applied to
any scattered $\Fq$-linear set of maximum rank in $\PG(1,q^n)$.\\
Indeed, by the $3$-transitivity of $\PGL(2, q^n )$  on the points of the finite projective line $\PG(1, q^n )$ and since the size of an $\Fq$-linear set of $\PG(1, q^n)$ is at most $(q^n - 1)/(q - 1)$, if
$q > 2$, then
any linear set of maximum rank is projectively equivalent to an $L_U$ such that
\begin{equation}\label{no3points}
\langle(1,0)\rangle_{\Fqn},\langle(0,1)\rangle_{\Fqn},\langle(1,1)\rangle_{\Fqn} \not \in  L_U.
\end{equation}
Hence,  there exists a linearized polynomial $f \in \tilde{\mathscr{L}}_{n,q}[X]$ such that $U=U_f$ and $\ker f=\{0\}=\ker(f-X)$.
If $f \in \tilde{\mathscr{L}}_{n,q}[X]$,  we will denote $\mathcal{D}_f=\{f(x)/x\colon x\in\F_{q^n}^*\}$ the set of the
non-homogeneous projective coordinates of the points belonging to the set $L_f$. Hence, by \eqref{no3points}, $0,1\notin \mathcal{D}_f$.\\
As the next result states, from each scattered polynomial a quasifield arises.

\begin{thm}\cite[Proposition 3.1]{Casarino_Longobardi_Zanella}\label{t:quasicorpo}
Let $f \in \tilde{\mathscr{L}}_{n,q}[X]$, $q >2$, be a scattered polynomial such that $0,1 \not \in \mathcal{D}_f$. Let $Q_f=\Fqn$ be endowed with the sum of $\Fqn$, and define
\begin{equation}\label{e:quasicorpo}
x\circ m=\begin{cases}
xm&\mbox{ if }m\notin \mathcal{D}_f,\\
h^{-1}f(hx)&\mbox{ if }m\in \mathcal{D}_f\mbox{ and }f(h)-mh=0,\ h\neq0
\end{cases}
\end{equation}
for any $x,m\in Q_f$.
Then $(Q_f,+,\circ)$ is a quasifield, and $K(Q_f)=\Fq$.
\end{thm}

Let $\cB_f=\cB(Q_f)$ and $\ptra_f$ be the planar spread of $Q_f^2$ associated with $Q_f$
and the related translation plane, respectively. Then, the elements of $\cB_f$ distinct from $B_\infty$ are of two types:
\begin{enumerate}
\item
 if $m\in\Fqn\setminus \mathcal{D}_f$, then $B_m=\langle(1,m)\rangle_{\Fqn}$;
\item
  if $m\in \mathcal{D}_f$ and $f(h)-mh=0$, $h\neq0$, then
  \[
    B_m=\{(h^{-1}x,h^{-1}f(x))\colon x\in\Fqn\}=h^{-1}U_f.
  \]
\end{enumerate}
In particular, $U_f=B_{f(1)}$ and $\{B_m\colon m\in \mathcal{D}_f\}$ and $\{\langle(1,m)\rangle_{\Fqn}\colon m\in \mathcal{D}_f\}$
are hyper-reguli covering the same vector set.

For instance, let $L_{s}^{1,n}$ be a linear set of pseudoregulus type. Then, $L_s^{1,n}$ is projectively equivalent to the linear set 
$L_{g}$ where $g=\omega X^{q^s}$ with $\omega \in \F_{q^n},  \mathrm{N}_{q^n/q}(\omega) \not \in \{0,1\}$, $\gcd(s,n)=1$. In this way, we get $0,1\notin \mathcal{D}_g$. In \cite[Section 3]{Casarino_Longobardi_Zanella}, Casarino \textit{et al.} underlined that the plane $\ptra_{g}$ associated with the scattered polynomial $g$ is the Andr\'{e }
$q$-plane  obtained by one $q^s$-Andr\'e net replacement.\\
However, if there is no need for a quasifield framework,  the assumption $0,1\notin \mathcal{D}_f$ can be deleted. Indeed, if $f\in \tilde{\mathscr{L}}_{n,q}[X]$ is a scattered $\F_q$-linearized polynomial, the collection $\mathcal{B}_f$ of the following subspaces of $\Fqn^2$:
\[
  \langle(1,m)\rangle_{\Fqn}\     (m\in\Fqn\setminus \mathcal{D}_f), \qquad   hU_f   \  (h\in\Fqn^*),
 \qquad\{0\}\times\Fqn.
\]
defines a planar spread and, hence a \emph{translation plane associated with $f$}, $\ptra_f=\ptra(\cB_f)$. However, the translation planes that come from scattered polynomials are neither coordinatizable over a semifield nor over a nearfield, see \cite[Corollary 4.4]{Casarino_Longobardi_Zanella}.\\
Let $f_1,f_2 \in \tilde{\mathscr{L}}_{n,q}[X]$ be scattered linearized polynomials, $q>3$ and let $\ptra_1:=\ptra_{f_1}$ and $\ptra_2:=\ptra_{f_2}$ be the translation planes associated. 
These are isomorphic if and only if $U_{f_1}$ and $U_{f_2}$
belong to the same orbit under the action of $\GaL(2,q^n)$, \cite[Theorem 4.2]{Casarino_Longobardi_Zanella}. This implies that any scattered $\Fq$-linear set $L=L_U$ of maximum rank in $\PG(1,q^n)$ gives rise to 
$c_\Gamma(L)$ pairwise nonisomorphic translation planes.\\
Finally, in \cite[Section 5]{Longobardi_Zanella2024} by investigating the affine collineation group of the translation plane associated with a scattered polynomial, the authors showed that, a maximum scattered linear set $L_f$ of the projective line $\PG(1,q^n)$, $n>2$ and $q>3$, is  of pseudoregulus type if and only if $\ptra_f$ is a (generalized) André plane \cite[Theorem 5.9]{Longobardi_Zanella2024}.

\section*{Acknowledgements
}
The author expresses gratitude to Prof. R. Trombetti for the valuable discussions that contributed to the drafting of this article.\\
Moreover, the author is supported by the Italian National Group for Algebraic
and Geometric Structures and their Applications (GNSAGA - INdAM) and by the European Union under the Italian National Recovery and Resilience Plan (NRRP) of NextGenerationEU, with particular reference to the partnership on \textit{"Telecommunications of the Future"}(PE00000001 - program "RESTART", CUP: D93C22000910001).

\vspace{0.5cm}

\noindent
Giovanni Longobardi,\\
Dipartimento di Matematica ed Applicazioni 'R. Caccioppoli'\\
Università degli Studi di Napoli Federico II,\\
Via Cintia, Monte S. Angelo I-80126 Napoli, Italy \\
\texttt{giovanni.longobardi@unina.it}

\end{document}